\documentclass[]{article}

\usepackage{amssymb}
\usepackage{cite} 
\usepackage{url}  
\usepackage{ifthen}  
\usepackage{multicol}   
\usepackage{geometry} 
\usepackage{amsmath}
\usepackage{array}
\usepackage{graphicx}
\usepackage[pdftex]{hyperref}
\usepackage[usenames]{color}
\usepackage{fancyhdr}
\usepackage{verbatim}
\usepackage{setspace}
\usepackage{graphicx}
\usepackage{tikz}
\usepackage{amsfonts}
\usepackage{bm}
\usepackage{mathtools}
\RequirePackage{hyperref}

\usepackage{amsthm}
\usepackage{color}

\newtheorem{theorem}{Theorem}[section]
\newtheorem{lemma}[theorem]{Lemma}

\theoremstyle{definition}
\newtheorem{definition}[theorem]{Definition}

\theoremstyle{remark}

\numberwithin{equation}{section}

\begin{document}

\title{Efficient computation of the Gr\"{u}nwald-Letnikov fractional diffusion derivative using adaptive time step memory}
\author{Christopher L. MacDonald, Nirupama Bhattacharya,\\Brian P. Sprouse, and Gabriel A. Silva\footnote{ Corresponding author: Prof. GA Silva\newline UC San Diego Jacobs Retina Center\newline 9415 Campus Point Drive\newline La Jolla, California 92037-0946\newline \newline Email: gsilva@ucsd.edu\newline\newline CLM and NB contributed equally.}\\\\Department of Bioengineeirng\\University of California, San Diego}
\date{}
\maketitle

\begin{abstract}
Computing numerical solutions to fractional differential equations can be computationally intensive due to the effect of non-local derivatives in which all previous time points contribute to the current iteration. In general, numerical approaches that depend on truncating part of the system history while efficient, can suffer from high degrees of error and inaccuracy. Here we present an adaptive time step memory method for smooth functions applied to the Gr\"{u}nwald-Letnikov fractional diffusion derivative. This method is computationally efficient and results in smaller errors during numerical simulations. Sampled points along the system's history at progressively longer intervals are assumed to reflect the values of neighboring time points. By including progressively fewer points backward in time, a temporally `weighted' history is computed that includes contributions from the entire past of the system, maintaining accuracy, but with fewer points actually calculated, greatly improving computational efficiency.
\end{abstract}

\newpage
\section{Introduction}
\label{sec:intro}
Fractional differential equations are an extension of integrals and derivatives of integer order.  Over the last few decades fractional differential equations have played an increasing role in  physics \cite{PhysRevE.55.6821,fellah2002application,metzler1999anomalous,soczkffiwicz2002application}, chemistry \cite{gorenflo2002time,seki2003fractional}, and engineering \cite{cushman36fractional,baeumer2001subordinated,mathieu2003fractional}.  
In particular, fractional calculus methods offer unique approaches for modeling complex and dynamic processes at molecular and cellular scales in biological and physiological systems, and there has been a recent interest in the use of fractional calculus methods in bioengineering and biophysical modeling \cite{Magin:2004p3137,Magin:2006uo}. These methods have been applied to modeling ultrasonic wave propagation in cancellous bone \cite{sebaa2006application}, bio-electrode behavior at the cardiac tissue interface \cite{magin2008modeling}, and describing spiny dendrite neurons using a fractional cable equation \cite{henry2008fractional}. Other work has shown that endogenous lipid granules within yeast exhibit anomalous subdiffusion during mitotic cell division \cite{selhuber2009variety}.

Here we investigate the numerical implementation and computational performance of a fractional reaction-diffusion equation. The continuous diffusion equation has been the focal point of transport modeling in physical and biological systems for over a hundred years, since first proposed by Adolf Fick in 1855. The practical motivation for the present work arose from considering cell signaling data between Muller neural glial cells in the neural sensory retina. The neural retina is a direct extension and part of the brain itself. Muller cells in the retina can communicate in a paracrine fashion by secreting adenosine triphosphate (ATP) that diffuses into the extracellular space that then triggers intracellular calcium waves. In some cases these signaling events can produce long range regenerative intercellular signaling in networks of cells; see \cite{Newman:2001p1009,Metea:2006bz,Yu:2009p3681}. In particular, from experimental data published in \cite{Newman:2001p1009}, we qualitatively observed a peaked diffusion profile for ATP where the central cusp-like shape persisted over a longer period of time than would be expected for Gaussian diffusion; this is a typical characteristic of anomalous subdiffusion \cite{metzler2000random}. Therefore we suspected that the diffusion of ATP in this physiological system was subdiffusive in nature. Neurobiologically, this could ultimately have an important effect on underlying physiology and cellular signaling. Motivated by these observations, we decided to explore ways of simulating and exploring such subdiffusive dynamics using numerical methods amenable to experimental methods and data. Modeling subdiffusive neurophysiological transport processes necessitates the use of fractional differential equations or other related objects. However, the numerical implementation of such equations involves non-local fractional derivative operators that require taking into account the entire past history of the system in order to compute the state of the system at the current time step. This produces a significant computational burden that limits the practicality of the models. In the present paper we introduce an `adaptive time step memory' approach for computing the contribution of the memory effect associated with the history of a system in a way that makes it both numerically and resource (i.e. computer memory) efficient and robust while maintaining accuracy. While our algorithms can be applied to any diffusion application modeled as a fractional process, they offer particular value for modeling complex processes with long histories, such as lengthy molecular or cellular simulations.

Approaches for numerically approximating solutions to fractional diffusion equations have been extensively studied \cite{chen2008finite,langlands2005accuracy,liu2007stability,mainardi2006sub,mclean2009convergence,podlubny2009matrix,yuste2003explicit}, but in general there is always a trade off between computational efficiency, complexity, and the accuracy of the resultant approximations. There has also been a considerable amount of work done on developing fast convolution quadrature algorithms, which is relevant to fractional differential equations because the non-local nature of fractional calculus operators results in either a continuous or discrete convolution of some form. In particular, Lubich and others \cite{Lubich-disc-fraccal1986,Lubich-ConvQuadDiscOperationalCalcI,Lubich-NonreflectingBoundary2002,Lubich-FastOblConvQuad2006,Lubich-AdaptiveFastOblConvEquations2008} have built a generalized and broadly applicable convolution quadrature framework to numerically approximate a continuous convolution integral with a wide array of possible convolution kernel functions (including oscillatory kernels, singular kernels, kernels with multiple time scales, and unknown kernels with known Laplace transforms), while achieving good algorithmic performance with respect to complexity and memory requirements. However, their algorithms are necessarily very complicated in order to handle a wide range of functions and expressions in the convolution while limiting storage requirements and scaling of arithmetic operations. They involve approximating a continuous convolution based on numerical inversions of the Laplace transform of the convolution kernel function using contour integrals discretized along appropriately chosen Talbot contours or hyperbolas. The scaling of the number of required arithmetic operations involved in the convolution, and overall memory requirements, are both reduced by splitting up the continuous convolution integral or discrete convolution quadrature summation into a carefully chosen series of smaller integrals or summations, and solving a set of ordinary differential equations one time step at a time without storing the entire history of the solutions. For methods explicitly involving a discrete convolution summation, the quadrature weights are calculated using the Laplace transform of the original kernel and linear multistep methods for solving differential equations. FFT techniques can be applied to calculate these weights simultaneously and further increase the efficiency of the algorithm \cite{Lubich-ConvQuadDiscOperationalCalcII}. These authors have demonstrated the success of their framework by simulating various differential equations involving convolutions, including a one-dimensional fractional diffusion equation \cite{Lubich-FastOblConvQuad2006}. In \cite{Li-FastTimeStepFracIntegrals2010}, Li takes an alternative approach from this framework and focuses on a fast time-stepping algorithm specifically for fractional integrals by constructing an efficient $Q$ point quadrature, but does not focus on how this fits into numerical algorithms representing larger mathematical models.

The methods we introduce here are also efficient, significantly reducing simulation times, computational overhead, and required memory resources, without significantly affecting the computed accuracy of the final solution. However, in contrast to broad generalized frameworks, our algorithms are focused on solving fractional differential equations involving nonlocal fractional derivative operators. We approximate the fractional derivative based on the Gr\"{u}nwald-Letnikov definition instead of pursuing a general quadrature formula approximation to the Riemann-Liouville integral definition of the fractional derivative. The Gr\"{u}nwald-Letnikov derivative is used in many numerical schemes for discretizing fractional diffusion equations from a continuous Riemann-Liouville approach \cite{chen2009numerical,yuste2003explicit,podlubny1999fractional} and involves a discrete convolution between a `weight' or coefficient function and the function for which we are interested in taking the derivative. The mathematics and theory of this form of a weighting function are well established in the literature \cite{podlubny1999fractional,Lubich-disc-fraccal1986}. Building on this foundation avoids the need for domain transformations, contour integration or involved theory. Our algorithms were specifically developed for real world applied diffusion and transport physics and engineering problems, often where there are practical measurement limitations to the types and quality (e.g. granularity) of the data that can be collected or observed. For situations such as these, our methods for handling discrete convolutions associated with fractional derivatives are more intuitive and accessible than other generalized mathematical methods, where it is often not obvious or clear how to implement and apply a generalized approach to a specific physical problem under a defined set of constraints. The approaches we introduce here can be incorporated with various combinations of finite difference spatial discretizations and time marching schemes of a larger mathematical model in a straightforward way.

The increased efficiency and memory savings in the approaches we describe here lie in the way the discrete summation is calculated. The conceptual basis that results in a saving of computational time without a huge tradeoff in accuracy is in the interpretation of the ‘weight’ function as a measure of the importance of the history of a system. The more recent history of the system is more important in determining the future state of the system, and therefore we make use of an `adaptive time step memory' method by changing the interval of the backwards time summation in the Gr\"{u}nwald-Letnikov fractional derivative. Instead of incorporating every previous time point into the most current calculation, only certain points are recalled based upon their proximity to the current point in time, weighted according to the sparsity of the time points. This substantially reduces the computational overhead needed to recalculate the prior history at every time step, yet maintains a high level of accuracy compared to other approximation methods that make use of the Gr\"{u}nwald-Letnikov definition. We consider two adaptive step approaches - one based on an arithmetic sequence and another based on a power law, that when combined with a linked-list computational data structure approach yield $\mathcal O(log_2(N))$ active memory requirements. This is a significant improvement over keeping all $N$ steps of the system's history. We compare our adaptive time step approach with the `short memory' method described by Volterra \cite{volterra1931} and Podlubny et al \cite{podlubny1999fractional}, and examine differences in simulation times and errors under similar conditions for both. In the last section we sketch out a `smart' adaptive memory method based on a continuous form of the Gr\"{u}nwald-Letnikov that will be able to accommodate more dynamic past histories, i.e., histories with sharp and abrupt changes relative to the time step being considered, that produce larger error in the discrete methods we discuss. Finally, we note that the scope of the present work focuses on a detailed development of the methods and algorithms. A full theoretical analysis of stability and algorithmic complexity is the topic of a companion paper that will follow this one.

\section{Fractional diffusion and the Gr\"{u}nwald-Letnikov derivative}

\subsection{Derivation of the fractional diffusion equation}
\label{sec:fracdiffderivation}
We begin by considering the standard diffusion equation with no reaction term

\begin{equation}
\label{eq:regulardiff}
\frac{\partial{u}}{\partial{t}}=\alpha_s\nabla^2u,
\end{equation}

\noindent  where $\alpha_{s}$ is the standard diffusion coefficient given in
units $\mbox{distance}^{2}/\mbox{time}$.  The time-fractional version of equation \ref{eq:regulardiff} is given by

\begin{equation}
\label{eq:fracdiffeq}
\frac{\partial{^{\gamma}u}}{\partial{t^\gamma}}=\alpha\nabla^2u.
\end{equation}

\noindent where here the diffusion coefficient $\alpha$ has fractional order units of $\mbox{distance}^{2}/\mbox{time}^{\gamma}$. This is the simplest form of the fractional diffusion equation.  The real values taken by $\gamma$ in equation \ref{fracdiffeq} dictate the diffusion regimes obtained by the equation.  Values of $0<\gamma<1$ result in subdiffusion, characterized by a slow diffusion profile relative to regular diffusion. When $\gamma>1$, an oscillatory component emerges in the diffusion profile resulting in superdiffusion.  As required, when $\gamma=1$ the fractional diffusion equation reduces to standard diffusion that results in a Gaussian profile.  When $\gamma=2$ it reduces to the standard wave equation. As a point of discussion, we note that there no obvious theoretical (mathematical) limits to the value  $\gamma$ can take. However, there necessarily exist practical limits on $\gamma$ when considering stability and maintaining order of accuracy for the numerical implementation and algorithms of the fractional differential equation. Here we only concern ourselves with a physically plausible interpretation behind the equation. We know what subdiffusion and superdiffusion correspond to physically, but, so far as we are aware, there is no straight forward meaning of equation \ref{eq:fracdiffeq} when $\gamma > 2$ that we are aware of. This may be a shortcoming of our current physical understanding in the sense that the physics underlying such cases are not yet well understood or have not yet been discovered to be connected to particular real world dynamical processes. So although theoretically there is no apparent limit, there is a practical numerical implementation limit that will depend on a desired order of accuracy and computational efficiency. As we eluded to in the Introduction, we leave a full theoretical analysis of the stability and complexity of our methods to a follow up companion paper.

Using the relation
\begin{equation}
\frac{\partial{u}}{\partial{t}}=\frac{\partial{^{1-\gamma}}}{\partial{t^{1-\gamma}}}\frac{\partial^\gamma u}{\partial t^\gamma}
\end{equation}
and substituting into equation \ref{fracdiffeq} yields

\begin{subequations}\label{eq:diffusioneq}
\begin{equation}
\frac{\partial{u}}{\partial{t}}=\frac{\partial{^{1-\gamma}}}{\partial{t^{1-\gamma}}}\alpha\nabla^2u
\end{equation}
\begin{equation}
\label{eq:fracdiffnoreact}
\frac{\partial{u}}{\partial{t}}=\alpha D^{1-\gamma}\nabla^2u
\end{equation}
\end{subequations}
where $D^{1-\gamma}$ denotes the fractional derivative operator which we define and discuss at length in the next section.

If we take into account consumption or generation processes, then we can rewrite the diffusion equation as

\begin{equation}
\label{eq:fracdifffull}
\frac{\partial{u}(\vec x,t)}{\partial{t}}={\alpha}D^{1-\gamma}\nabla^2u(\vec x,t)+f(u), ~~~\vec x=\left\{x_1,x_2...x_N\right\}, ~~~u\left(\vec{x},0\right)=u_{0}\left(\vec{x}\right)
\end{equation}

\noindent where $\vec x$ is an $N$ dimensional vector and \emph{f(u)} represents a consumption or generation term. For the specific examples we describe in this paper we implement the simplest Dirichlet boundary conditions where all boundaries are set to zero. However, our method can handle various other boundary conditions, including Dirichlet conditions set to any values, Neumann conditions, and others, in a straight forward manner. In the following section we consider the fractional derivative operator $D^{1-\gamma}$.

\subsection{Definition of the Gr\"{u}nwald-Letnikov derivative}
The $a^{th}$ order fractional derivative of a function, denoted by $D^{a}f$, extends the order of the differential operator from the set of integers to the set of real numbers. The operator can be mathematically defined in several ways, including standard descriptions like the Riemann-Liouville definition, the Gr\"{u}nwald-Letnikov definition, and others. For numerical simulations, we find the Gr\"{u}nwald-Letnikov derivative to be convenient, since it is based on the standard differential operator but made applicable to arbitrary order $a$ with a discrete summation and binomial coefficient term:

\begin{equation}
\label{eq:naturalGrunwald}
D^{a}f(x)=\lim_{h\rightarrow0}h^{-a}\sum_{m=0}^{x/h}(-1)^{m} {a\choose m}f(x-mh).
\end{equation}

\noindent \\Expanding the binomial coefficient yields

\begin{equation}
 {a\choose m}=\frac{a!}{m!(n-m)!}
\end{equation}

\noindent \\and expressing the factorial components of the binomial coefficient by the gamma function gives

\begin{equation} 
{a\choose m}=\frac{\Gamma(a+1)}{m!\Gamma(a+1-m)}.  \label{eq:gammabi}
\end{equation}

\noindent Combining equations \ref{eq:naturalGrunwald} and \ref{eq:gammabi} yields the Gr\"{u}nwald-Letnikov derivative for real numbers:

\begin{equation} \label{eq:realGrunwald}
D^{a}f(x)=\lim_{h\rightarrow0}h^{-a}\sum_{m=0}^{x/h}\frac{(-1)^{m}\Gamma(a+1)}{m!\Gamma(a+1-m)}f(x-mh)~~~~~a\in\mathbb{R},~a\neq-\mathbb{N}_1.
\end{equation}
\\
\indent It should be noted that computation of the Gr\"{u}nwald-Letnikov derivative requires knowledge of the entire past history of the system due to the summation operator. This means that this fractional derivative, unlike standard derivatives, is no longer local but global.  Computing the value of a real derivative requires knowledge of the system's past history. The derivative must take into account all past points from $m=0$, the current point, all the way back to the beginning point $m=x/h$. This requirement places significant computational demands on the numerical implementation of the Gr\"{u}nwald-Letnikov derivative, and is the principle technical motivation for the results presented here. 

In the rest of the paper we explore the implementation of equation \ref{eq:realGrunwald} to a diffusion regime constrained to $0<\gamma \equiv a \leq2$ (as explained in section \ref{sec:fracdiffderivation}): 

\begin{equation}
\label{eq:grunwaldlet}
D^{\gamma}f(x)=\lim_{h\rightarrow0}h^{-\gamma}\sum_{m=0}^{x/h} \frac{(-1)^{m}\Gamma(\gamma+1)}{m!\Gamma(\gamma+1-m)}f(x-mh) \text{ for }0<\gamma\leq2.
\end{equation}

Applying equation \ref{eq:grunwaldlet} to the fractional diffusion equation (eq. \ref{eq:fracdiffnoreact}) yields:  

\begin{equation}
 \frac{\partial{u}}{\partial{t}}=\lim_{\tau\rightarrow0}\tau^{\gamma-1}\alpha\sum_{m=0}^{t/\tau}\frac{(-1)^m\Gamma(2-\gamma)}{m!\Gamma(2-\gamma-m)} \nabla^2u(t-m\tau)
\label{eq:grunlet_equation}
\end{equation}

\noindent\\ where $t$ represents instantaneous time, and $\tau$ is the time step.

\noindent\\ Next, define a function $\psi(\gamma,m)$ such that

\begin{equation}
\label{eq:psidefinition}
\psi(\gamma,m)=\frac{(-1)^m\Gamma(2-\gamma)}{m!\Gamma(2-\gamma-m)}.
\end{equation}

\noindent\\ Given that the standard definition of the gamma function implies $\Gamma(a)=\frac{1}{a}\Gamma(a+1)$, equation \ref{eq:psidefinition} can be reduced to a recursive relationship.  Consider the first four terms in the chain from $m=0$ to $m=3$:

\begin{eqnarray*}
\psi(\gamma,0)&=&~~~~~\frac{(-1)^{0}\cdot\Gamma(2-\gamma)}{(0!)\cdot\Gamma(2-\gamma-0)}=1\nonumber\\
\psi(\gamma,1)&=&~~~~~\frac{(-1)^{1}\cdot\Gamma(2-\gamma)}{(1!)\cdot\Gamma(2-\gamma-1)}=\frac{(-1)\cdot(-1)^{0}\cdot\Gamma(2-\gamma)}{(1)\cdot(0!)\cdot\frac{1}{(2-\gamma-1)}\cdot\Gamma(2-\gamma-0)}\nonumber\\
\psi(\gamma,2)&=&~~~~~\frac{(-1)^{2}\cdot\Gamma(2-\gamma)}{(2!)\cdot\Gamma(2-\gamma-2)}=\frac{(-1)\cdot(-1)^{1}\cdot\Gamma(2-\gamma)}{(2)\cdot(1!)\cdot\frac{1}{(2-\gamma-2)}\cdot\Gamma(2-\gamma-1)}\nonumber\\
\psi(\gamma,3)&=&~~~~~\frac{(-1)^{3}\cdot\Gamma(2-\gamma)}{(3!)\cdot\Gamma(2-\gamma-3)}=\frac{(-1)\cdot(-1)^{2}\cdot\Gamma(2-\gamma)}{(3)\cdot(2!)\cdot\frac{1}{(2-\gamma-3)}\cdot\Gamma(2-\gamma-2)}\nonumber\\
\end{eqnarray*}

\noindent Examining this system of equations suggests that for any $\psi(\gamma,m)$:

\begin{equation}
 \psi(\gamma,m)=\frac{(-1)^{m}\cdot\Gamma(2-\gamma)}{(m!)\cdot\Gamma(2-\gamma-m)}=\frac{(-1)\cdot(-1)^{(m-1)}\cdot\Gamma(2-\gamma)}{(m)\cdot(m-1)!\cdot\frac{1}{(2-\gamma-m)}\cdot\Gamma(2-\gamma-(m-1))}.
\end{equation}

\noindent 
Because the function $\psi(\gamma,m)$ is dependent on $\psi(\gamma,m-1)$, an iterative relationship forms that scales by $\frac{-(2-\gamma-m)}{m}$:

\begin{equation}
\label{eq:iterativepsi}
 \psi(\gamma,m)=-\psi(\gamma,m-1)\frac{2-\gamma-m}{m}.
\end{equation}

This recursive function is valid for all $\gamma$ including subdiffusion, standard diffusion and superdiffusion, so this equation is general over all regimes. Because the entire history of the system must be taken into account when computing the current time step, $\psi(\gamma,m)$ is used for many $m$ values, many times over the course of the simulation, and can be precomputed for values of $m=0$ to $m=N$, where $N$ is the total number of time points, resulting in a significant savings in computational performance. A similar simplification has been previously discussed and used in \cite{podlubny1999fractional}.  Taking $\psi(\gamma,m)$ into account yields the final form of the fractional diffusion equation, given by

\begin{equation}
\label{eq:grunwald-psieq}
\frac{\partial{u}}{\partial{t}}=\lim_{\tau\rightarrow0}\tau^{\gamma-1}\alpha\sum_{m=0}^{t/\tau}\psi(\gamma,m) \nabla^2u(t-m\tau).
\end{equation}

We can also arrive at equation \ref{eq:grunwald-psieq} by an alternative approach. Begin by considering again the Gr\"{u}nwald-Letnikov derivative in terms of binomial notation:

\begin{equation}
D^{1-\gamma}f(t)=\lim_{\tau\rightarrow0}\tau^{\gamma-1}\sum_{m=0}^{t/\tau}(-1)^m {{1-\gamma}\choose m}f(t-m\tau).
\label{eq:grunletdefbi}
\end{equation}

\noindent Applying eq. \ref{eq:grunletdefbi} to the fractional diffusion equation (eq. \ref{eq:fracdiffnoreact}) yields

\begin{equation}
 \frac{\partial{u}}{\partial{t}}=\lim_{\tau\rightarrow0}\tau^{\gamma-1}\alpha\sum_{m=0}^{t/\tau}(-1)^m {1-\gamma\choose m} \nabla^2u(\vec x, t-m\tau)
\label{eq:grunlet_equation2}
\end{equation}
Knowing the gamma function definition of the factorial operator, $a!=\Gamma(a+1)$, the binomial coefficient can be expanded to accommodate real numbers (also \emph{c.f.} equation \ref{eq:gammabi}):

\begin{equation}
\label{eq:binom}
 {{1-\gamma}\choose m}=\frac{(1-\gamma)!}{m!(1-\gamma-m)!}=\frac{\Gamma(2-\gamma)}{m!\Gamma(2-\gamma-m)}.
\end{equation}
Note how combining equations \ref{eq:grunlet_equation2} and \ref{eq:binom} yields equation \ref{eq:grunlet_equation} as required. 

Next, defining the function $\psi(\gamma,m)$ (equation \ref{eq:psidefinition}) using binomial notation gives
\begin{equation}
\label{psidefinition}
 \psi(\gamma,m)=(-1)^m {1-\gamma\choose m}
\end{equation}
and substituting the relation $\Gamma(a)=(a-1)\Gamma(a-1)$ into eq. \ref{eq:binom} results in

\begin{equation}
 \psi(\gamma,m)=\frac{(-1)^{m-1}\Gamma(2-\gamma)}{(m-1)!\Gamma(2-\gamma-(m-1))}\frac{-1}{m\frac{1}{2-\gamma-m}}
\end{equation}

\noindent yielding the iterative relationship in equation \ref{eq:iterativepsi} above:

\begin{equation} \nonumber
 \psi(\gamma,m)=-\psi(\gamma,m-1)\frac{2-\gamma-m}{m}.
\end{equation}
With the substitution of $\psi(\gamma,m)$ into eq. \ref{eq:grunlet_equation2} we once again recover the time-fractional diffusion equation in the form

\begin{equation} \label{eq:fulldiff}
\frac{\partial{u}}{\partial{t}}=\lim_{\tau\rightarrow0}\tau^{\gamma-1}\alpha\sum_{m=0}^{t/\tau}\psi(\gamma,m) \nabla^2u(\vec x, t-m\tau).
\end{equation}

\subsection{Discretization of the fractional diffusion equation.}
Using 
\begin{equation}
 \nabla^2u(\vec x, t-m\tau)=\frac{\partial^2u(\vec x,t-m\tau)}{\partial x_1^2}+\frac{\partial^2u(\vec x,t-m\tau)}{\partial x_2^2}
\end{equation}
as the formulation in two spatial dimensions, one can discretize the function into a finite difference based FTCS scheme (forward time centered space) on a grid $u_{j,l}^k$ (where $k=t/\Delta_t,j=x_1/\Delta_{x},l=x_2/\Delta_{x}$, $\Delta_x$ is the grid spacing in both directions assuming an equally spaced grid, and $\Delta_t$ is the time step), using the relations \cite{press1992numerical}

\begin{eqnarray} \label{eq:ftcs}
\frac{\partial^2u(\vec x,t-m\tau)}{\partial x_1^2}&=&\frac{u_{j+1,l}^{k-m}-2u_{j,l}^{k-m}+u_{j-1,l}^{k-m}}{\Delta_{x}^2}\nonumber \\
\frac{\partial^2u(\vec x,t-m\tau)}{\partial x_2^2}&=&\frac{u_{j,l+1}^{k-m}-2u_{j,l}^{k-m}+u_{j,l-1}^{k-m}}{\Delta_{x}^2}\nonumber \\
\frac{\partial u(\vec x,t)}{\partial t}&=&\frac{u^{k+1}_{j,l}-u^{k}_{j,l}}{\Delta_t}.
\end{eqnarray}

In the discrete limit where $\tau\rightarrow\Delta_t$, we approximate the fractional diffusion equation (equations \ref{eq:grunwald-psieq} and \ref{eq:fulldiff}) with the following finite difference expression, which has a first-order approximation $O(\tau)$ (Chapter 7 in \cite{podlubny1999fractional}, \cite{Lubich-disc-fraccal1986}):

\begin{equation}
\label{frac_norxn}
\frac{u^{k+1}_{j,l}-u^{k}_{j,l}}{\Delta_t}=\alpha\frac{\Delta_t^{\gamma-1}}{\Delta_{x}^2}\sum_{m=0}^k\psi(\gamma,m)\delta^{k-m}_{j,l}
\end{equation}

\noindent where $\delta^{k-m}_{j,l}$ is the finite difference kernel given by

\begin{equation}
\label{eq:deltasimp}
\delta^{k-m}_{j,l}=\left(u_{j+1,l}^{k-m}+u_{j-1,l}^{k-m}-4u_{j,l}^{k-m}+u_{j,l+1}^{k-m}+u_{j,l-1}^{k-m}\right).
\end{equation}

Adding a consumption/generation term is straightforward in this implementation. For example, take an exponential decay term given by

\begin{equation}
\label{expdecay}
 \frac{\partial{u}}{\partial{t}}=-\beta u
\end{equation}
with the complementary finite difference relation

\begin{equation}\label{eq:findif}
 \frac{u^{k+1}_{j,l}-u^{k}_{j,l}}{\Delta_t}=-\beta u_{j,l}^{k}.
\end{equation}
Incorporating eq. \ref{expdecay} into the form of eq. \ref{eq:fracdifffull} results in 

\begin{equation}
 \frac{\partial{u}}{\partial{t}}={\alpha}D_{t}^{1-\gamma}\nabla^2u-\beta u,
\end{equation}
which gives the full finite difference implementation in two dimensions

\begin{equation}
\label{eq:frac_final}
\frac{u^{k+1}_{j,l}-u^{k}_{j,l}}{\Delta_t}=\alpha\frac{\Delta_t^{\gamma-1}}{\Delta_x^2}\sum_{m=0}^k\psi(\gamma,m)\delta^{k-m}_{j,l} -\beta u_{j,l}^{k}
\end{equation}

\begin{figure}
\begin{center}
\includegraphics[width=5in]{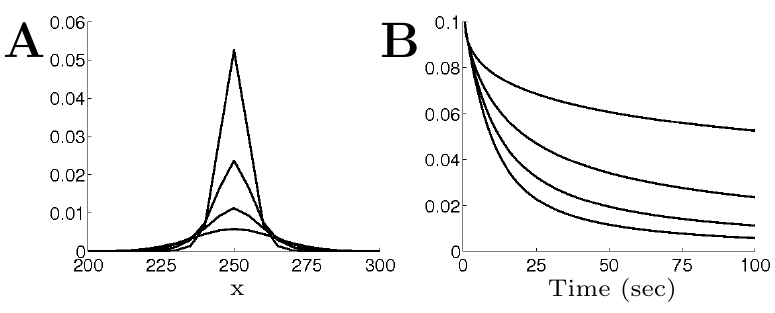}    
\caption{Simulation results for $\gamma=0.5,~0.75,~0.9,~1.0$ (for traces from top to bottom) in one dimensional space (panel A) and time (panel B). As the subdiffusion regime is approached the profile becomes more and more hypergaussian.} 
\label{diffresults}
\end{center}
\end{figure}

Figure \ref{diffresults} shows the results of four simulations for  different values of $\gamma$. Simulations were run on a 100 x 100 grid with initial conditions $U_{50,50}^0$ = 0.1, $U_{51,50}^0=U_{50,51}^0=U_{49,50}^0=U_{50,49}^0$ = 0.05 and zero elsewhere. Dirichlet boundary conditions were implemented and the simulation edge set to zero. Simulations were run for $t=100$ seconds with $\alpha=1$, $\beta=0$, and $\Delta_x=5$.

\section{Adaptive time step memory as an arithmetic sequence}
\subsection{Derivation}
In eq. \ref{eq:frac_final} each iteration requires the re-calculation and summation of every previous time point convolved with $\psi(\gamma,m)$.  This becomes increasingly cumbersome for large times, which require significant numbers of computations and memory storage requirements. To address this, Podlubny et. al \cite{podlubny1999fractional} introduced the `short memory' principle which assumes that for a large $t$ the role of the previous steps or `history' of the equation become less and less important as the convolved $\psi(\gamma,m)$ shrinks towards zero. This would then result in approximating eq. \ref{eq:frac_final} by truncating the backwards summation, only taking into account times on the interval $[t-L,t]$ instead of $[0,t]$, where $L$ is defined as the `memory length' (eq. 7.4 in \cite{podlubny1999fractional}; Fig. \ref{adaptivetime}). While computationally efficient, this approach leads to errors in the final solution since not all points are counted in the summation. Despite the resultant errors, this numerical method represents a powerful and simple approach for providing a reasonable trade off between computational overhead and numerical accuracy. In the context of the implementation derived here, it would result in a discretization scheme given by 

\begin{equation}
\label{frac_shortmemory}
\frac{u^{k+1}_{j,l}-u^{k}_{j,l}}{\Delta_t}=\alpha\frac{\Delta_t^{\gamma-1}}{\Delta_x^2}\sum_{m=0}^{min(L/\Delta_t,k)}\psi(\gamma,m)\delta^{k-m}_{j,l} -\beta u_{j,l}^{k}
\end{equation}

As an alternative to the method of Podlubny, Ford and Simpson proposed an alternative `nested mesh' variant that gives a good approximation to the true solution at a reasonable computational cost \cite{ford2001numerical}.  However, this method is exclusive to the Caputo fractional derivative. Here we introduce an approach that is applicable to the more general Gr\"{u}nwald-Letnikov definition.  Like these other methods, it also shortens computational times but at the same time results in much greater accuracy than the use of `short memory.'  We achieve this by introducing the concept of an `adaptive memory' into the Gr\"{u}nwald-Letnikov discretization.  

The underlying principle of the adaptive time approach is that relative to the current time point previous time points contribute different amounts to the summation. Values relatively closer to the current time point will have a greater contribution to the current numerical calculation than values many time points back due to the multiplier $\psi(\gamma,m)$.  For smooth functions, as $m$ increases and $\left|\psi(\gamma,m)\right|$ decreases, neighboring points in the summation exhibit only small differences. Consequently, one can take advantage of this and utilize an `adaptive memory' approach in which neighboring values at prior time points are grouped together in defined increments and the median taken as a representative contribution for the increment weighted according to the length of the increment to account for the skipped points. This results in fewer time points that need to be computed in a summation. Algorithmically, for an arbitrary number (define this as parameter $a$) of time steps back from the current time point $k$ for which the history of the system is being computed, consider an initial interval $[0,a]$ for which all time points within this interval are used in the summation and therefore contribute to the the Gr\"{u}nwald-Letnikov discretization. Let subsequent time intervals become longer the further away from $k$ they are and within them only every other $d$ time points are used in the summation term, i.e., only the median values of a temporally shifting calculation increment of length $d$ along the current interval are considered. As subsequent intervals become longer, so does $d$, thereby requiring less points to compute (Fig. \ref{adaptivetime}). 

\begin{definition} Let $k$ be the current iterative time step in a reaction diffusion process for which a Gr\"{u}nwald-Letnikov discretization is being computed. Consider an arbitrary time point $a$ in the history of the system backwards from $k$. For $i\in \mathbb{N}_1,i\neq 1$, define an interval of this past history by
\begin{equation}\label{eq:interval}
I=[a^{i-1}+i,a^i]
\end{equation}
where $\mathbb{N}_1$ represents the set of natural numbers beginning from one.  Given how the indices $i$ are defined, the very first interval backwards from $k$ is independent of equation \ref{eq:interval} and is given by $[0,a]$. This is considered the base interval. Subsequent intervals are defined as a function of this base, i.e., as a function of $a$ and determined by eq. \ref{eq:interval}. Let $i_{max}$ be the value of $i$ such that $k \in I_{max}=[a^{i_{max}-1}+i_{max},a^{i_{max}}]$. The complete set of intervals then is defined as $\zeta=\{I=[a^{i-1}+i,a^i]:i\in \mathbb{N}_1,i\neq 1, i\le i_{max}\}$. 
\end{definition}
\begin{definition}
For the set of intervals $\zeta$ as defined in Definition 3.1, $D=\{d=2i-1:i\in \mathbb{N}_1,i\neq 1\}$ is the set of distances $d$ by which the corresponding intervals in $\zeta$ are sampled at.
\end{definition}
\begin{theorem}
In general, for two-dimensional diffusion without  consumption or generation terms for any interval as defined in Definition 3.1, the Gr\"{u}nwald-Letnikov discretization with adaptive time step memory is given by 
\begin{align}\label{eq:adapttime}
& \frac{u^{k+1}_{j,l}-u^{k}_{j,l}}{\Delta_t}=\alpha\frac{\Delta_t^{\gamma-1}}{\Delta_x^2}\Bigg[\sum_{n=0}^{a}\psi(\gamma,n)\delta^{k-n}_{j,l}+\cdots \\ \nonumber
&\sum_{i=2}^{i_{max}}\sum_{m_i=a^{i-1}+i}^{a^i}(2i-1)\psi(\gamma,m_i)\delta^{k-m_i}_{j,l}+\sum_{p=m_{max}+{i_{max}}}^{k}\psi(\gamma,p)\delta^{k-p}_{j,l}\Bigg]
\end{align}
where $p \in \mathbb{N}_1$, and for each $i$ (i.e. for each interval) $M=\{m_i=a^{i-1}+(2i-1)\eta-i+1:\eta\in\mathbb{N}_1~\&~m_i\le m_{max}\}$ is the set of time points over which $\psi(\gamma,m)\delta^{k-m}_{j,l}$ is evaluated. Since the time point $k$ may be less than the full length of the last interval $I_{max}$, $|m_{max}| \le |k-i_{max}|$ represents the maximum value in $I_{max}$ that is evaluated, i.e. the last element in the set $M$ for $I_{max}$.
\end{theorem}
\begin{proof}The first summation represents the basis interval and the limits of the summation operator imply the contribution of every time point, i.e., $n\in\mathbb{N}_1$. For intervals beyond $a$: any arithmetic sequence defined by a recursive process $\nu_\eta=\nu_{\eta-1} + d$, $\eta\in\mathbb{N}_1$ for some distance $d$, the $\eta^{th}$ value in the sequence can be explicitly calculated as $\nu_\eta=\nu_1+(\eta-1)d$ given knowledge of the sequence's starting value $\nu_1$. For the set $\zeta$ this implies that $\nu_1=a^{i-1}+i$ and $d=2i-1$ for a given $i$. This then yields $\nu_\eta=a^{i-1}+i+(\eta-1)(2i-1)=a^{i-1}+(2i-1)\eta-i+1\mathrel{\mathop:}= m_i$ as required. The outer summation collects the summations of all intervals that belong to the set $\zeta$.  The last summation on the interval $[m_{max}+i_{max},k]$ ensures that the final point(s) of the backwards summation are still incorporated into the computation even if the arbitrarily chosen value of $a$ does not generate a final increment length that ends on $k$. 
\end{proof}

Note that $D$ is not explicitly needed for computing equation \ref{eq:adapttime} because the distances $d$ are implicitly taken into account by $\zeta$. Using the median value of each increment avoids the need for interpolation between time points. The implementation of the adaptive memory method described here is necessarily limited to smooth functions due to the assumption that the neighbors of the median values used in the summation do not vary much over the increment being considered. This method essentially allows for a contribution by all previous time points to the current calculation, yet reduces computational times by minimizing the total number of calculations. 
\begin{figure}
\begin{center}
\includegraphics[width=4in]{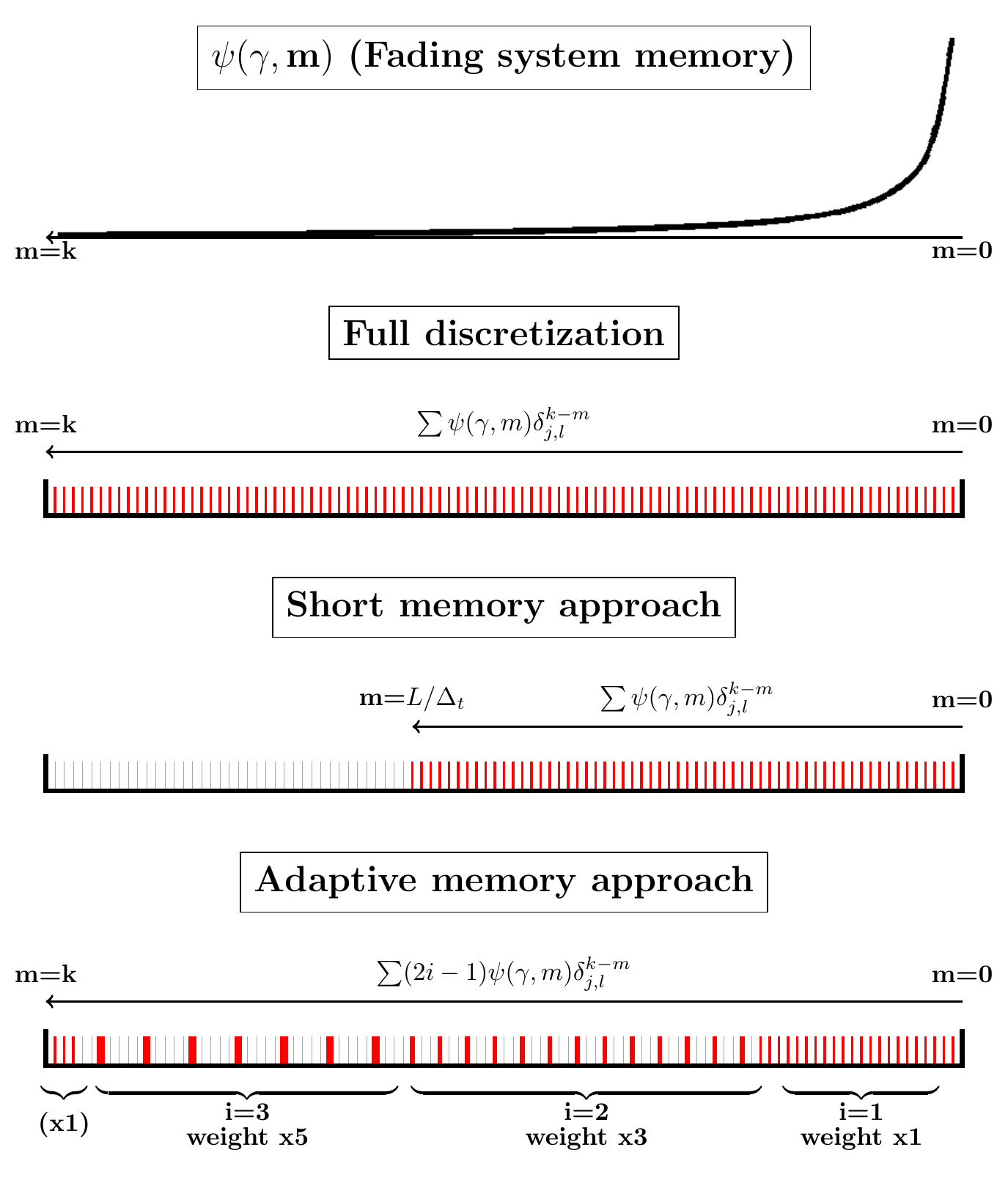}    
\caption{Short memory and adaptive memory methods for estimating the the Gr\"{u}nwald-Letnikov discretization. Both approximations rely on the sharply decreasing function $\left|\psi(\gamma,m)\right|$ as $m$ increases to omit points from the backwards summation. While short memory defines a sharp cut off of points, adaptive memory provides a weighted sampling of points for the entire history of the system. Points included in the computation by each method are highlighted in red. The shape of $\psi$ is different for $\gamma<1$ and $\gamma>1$, but the shape of $\left|\psi\right|$ remains a monotonically decreasing function (as $m$ increases) for both cases, and remains consistent with the principle that more recent time points contribute (whether positively or negatively) more to the solution at the next time step, than time points further back in the history of the system. See text for details.} 
\label{adaptivetime}
\end{center}
\end{figure}

\subsection{Comparison to short memory methods}
The results of using various $L$ (for short memory) and interval steps $a$ (for adaptive memory) are shown in Fig.~\ref{shortvsadaptive}. Increasing the values of $L$ and $a$ resulted in a decrease in the error of the estimated results but at the cost of increased computation times.  Conversely, decreasing the values of $L$ and $a$ resulted in a decrease in computation times, but at the expense of accuracy.  In all cases however, the adaptive memory method had a significantly lower error for the same simulation time, and also reached a clear asymptotic minimum error much faster than the minimum error achieved by the short memory method. In these simulations, $\alpha=1$, $\beta=0$, $\Delta_t=1$, $\Delta_x=10$, using a 20 x 20 grid, and ran for $t=1500$ where $U_{10,10}^0=10$. The error for the `short memory' method increased comparatively quickly and worsened as $\gamma\rightarrow1$. This was due to the fact that the evolution of the solution was initially quite fast at times near $t=0$, which were the first time points cut by the `short memory' algorithm. 

\begin{figure}
\begin{center}
\includegraphics[width=4in]{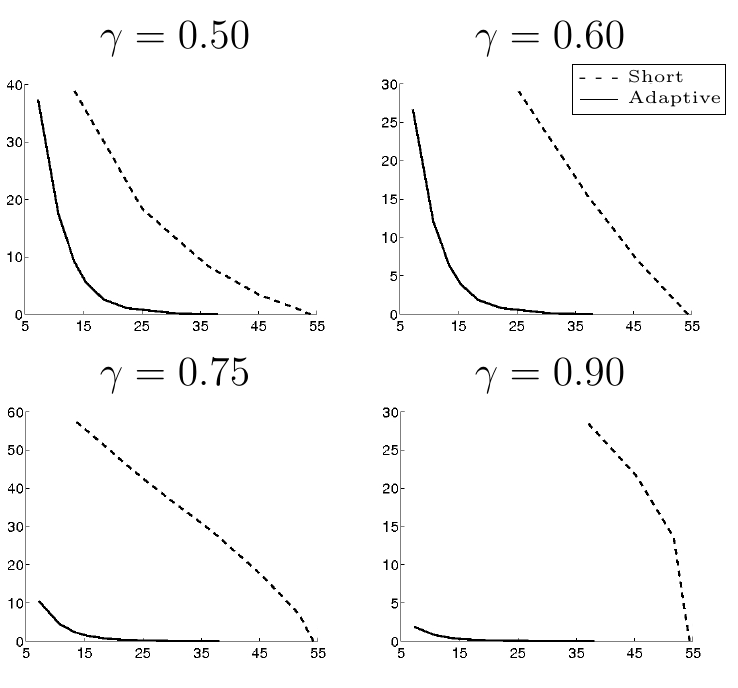}    
\caption{Comparison of the error between adaptive memory and the short memory as a function of the calculation time (x-axis: computation run time in seconds) expressed as a percentage error relative to the computed value for the full non-shortened approach (y-axis). Four different values of $\gamma$ are shown.} 
\label{shortvsadaptive}
\end{center}
\end{figure}


\section{Adaptive time step as a power law}
\label{fracminmem}
While computationally  efficient and intuitively straightforward, the major drawback to the adaptive memory method computed as an arithmetic sequence, is the necessary amount of allocated memory. The backwards time summation grid changes with every step such that every point in the history needs to be retained and stored during the course of a simulation. For high dimensional diffusion this limits the practical applicability of the method. For example, if one were to solve the three dimensional diffusion equation on just a 100 x 100 x 100 set of grid points, that would require the storage of 1,000,000 data points per time step, which over the course of a long simulation would overwhelm the memory capacity of most computers in addition to greatly slowing simulation times. The same memory issues arise when considering another popular approach used for solving the fractional diffusion equation, discussed in \cite{podlubny2009matrix}, that uses an implicit matrix solution to simultaneously solve all points in time using a triangular strip matrix formulation. This later approach is very powerful but does not take advantage of any short memory algorithms. 
 
In this section we improve on our results and develop a framework that uses a power law-based rule for eliminating past history points without sacrificing numerical accuracy. While the adaptive memory algorithm can be used with various distributions other than an arithmetic sequence, our motivation for choosing a power law-based distribution is that it matches the natural power law dynamics of the underlying mathematics associated with the time-fractional derivative, and it is also easy to model and implement. We therefore considered it a good starting point for improvements on the original arithmetic sequence version of the adaptive memory algorithm. We note however that other distributions can be combined with an adaptive step approach that would result in varying degrees of efficiency and accuracy. This will depend on how well they match the dynamics associated with time-fractional diffusion. In this section though we only consider a power law-based rule.

An adaptive memory approach applied to a power law distribution, in combination with a linked-list based method, minimizes the storage of the past history of the function. This results in decreasing the amount of total memory allocated to run a simulation of $N$ time steps from $\mathcal O(N)$ to $\mathcal O(log_2(N))$, resulting in a tremendous memory savings.
Given a simulation of $N$ time steps and number of grid points $X/\Delta_x$, where $X$ is the grid width, storing the entire past history would require $\frac{X}{\Delta_x}N$ points to be stored, which grows linearly with $N$. In contrast, an adaptive mesh with a power law growth in spacing (in this case $1,2,4 ...$), results in $\frac{X}{\Delta_x}log_2(N)$ points being stored in memory, which grows with the $log$ of the number of points $N$. The advantage of using such a power law scaling is that one can {\it a priori} calculate memory points which will not be used at every time step of the simulation and de-allocate the memory assigned to storing those time points. With the adaptive memory method, past points needed for each subsequent time step change with each iteration, necessitating that all time points be stored for future calculations. The use of a self similar power law allows the elimination of points that will not be needed at any future time in the calculation. A comparison of the implementation of this method with the full implementation using every time point, is shown in Figure \ref{Minmemtime}. 

\begin{figure}
\begin{center}
\includegraphics[width=4in]{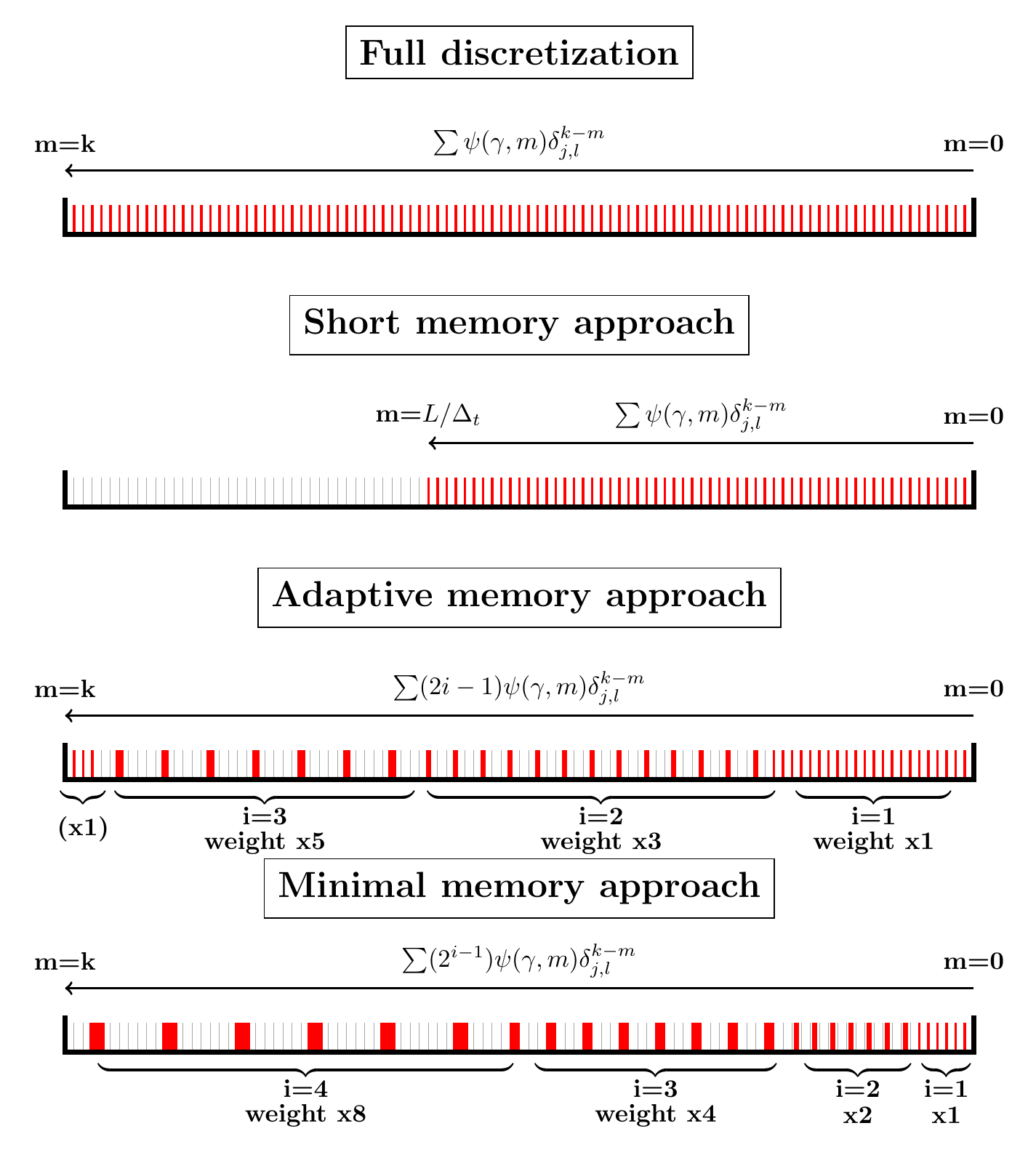}    
\caption{\footnotesize Minimal memory implementation} 
\label{Minmemtime}
\end{center}
\end{figure}

\subsection{Set theoretic implementation}
\label{MinMemMath}
In one spatial dimension, an FTCS discretization of the fractional diffusion equation on a grid $u_{j}^i$ where $i=t/\Delta_t,j=x_1/\Delta_{x}$, can be written as 

\begin{equation}
\label{eq:frac_minmem}
\frac{u^{i+1}_{j}-u^{i}_{j}}{\Delta_t}=\alpha\frac{\Delta_t^{\gamma-1}}{\Delta_{x}^2}\sum_{m=0}^i\psi(\gamma,m) \delta^{i-m}_j.
\end{equation}

(\emph{c.f.} equations \ref{eq:ftcs} to \ref{eq:deltasimp} above and \cite{press1992numerical}), where

\begin{equation}
 \delta^i_j=\left(u_{j+1}^{i}-2u_{j}^{i}+u_{j-1}^{i}\right)
\end{equation}

\begin{definition} Define a well-ordered set $U$ in time consisting of elements $u^i_j$ ordered by the point in time $i=t/\Delta_t$ at which that grid point occurred. The least elements in this set are then the points where $u^0_j$, and the greatest are the points $u^k_j$, where the current time point is $k=t_{current}/\Delta_t$. For all integers $A$ and $B$, if $A>B$, then $u^A_j>u^B_j$.
\end{definition}

Given the numerical scheme in \ref{eq:frac_minmem}, the set $U$ can be expressed as the recursive algorithm 
\begin{eqnarray}
u^{k+1}_{j}&=&u^{k}_{j}+\alpha\frac{\Delta_t^{\gamma}}{\Delta_{x}^2}\sum_{\{u^i_j\in U\}}\psi(\gamma,k-i)\delta^i_j.
\end{eqnarray}
where $u_j^0$ is the set of initial conditions at $t=0$.

Taking advantage of this result we can state the following lemma for the short memory approach:
\begin{lemma} \label{short_U}
Assume a short memory scheme. The elements $u^i<u^{k-\frac{L}{\Delta_t}}$ in $U$ for a time step $i$ are not used in future computations and can be removed from $U$.
\end{lemma}
The set $U$ can then become a list of sets $U^i$ with only the necessary elements to complete the recursive relation in each step $i$. 
\begin{proof}
The short memory approach shown in Figure \ref{Minmemtime} can be written as 
\begin{eqnarray}
u^{k+1}_{j}&=&u^{k}_{j}+\alpha\frac{\Delta_t^{\gamma}}{\Delta_{x}^2}\sum_{\{u^i_j\in U:u^i_j>u^{k-\frac{L}{\Delta_t}}\}}\psi(\gamma,k-i)\delta^i_j.\
\end{eqnarray}
This recursive relation can be written as 
\begin{eqnarray}
u^{k+1}_{j}&=&u^{k}_{j}+\alpha\frac{\Delta_t^{\gamma}}{\Delta_{x}^2}\sum_{\{u^i_j\in U^k\}}\psi(\gamma,k-i)\delta^i_j.\nonumber \\
U^{k+1}&:=&\{u^i\in U^k:u^i>u^{k-\frac{L}{\Delta_t}}\}+\{u^{k+1}\}
\end{eqnarray}
with $U^0:=\{u^0\}$. As the function evolves, elements in $U$ given by $u^i<u^{k-\frac{L}{\Delta_t}}$ will never be used again and can be removed.
\end{proof} 
Numerically only the set $U^k$ needs to be stored in memory for the current time point $k$ and all points after. As discussed above, by its construction adaptive memory time step as an arithmetic sequence necessitates a shifting of the calculated window of points, and as such the entire history of the system needs to be stored in memory for the calculation at all time points. However, we can construct an adaptive memory as a power law, such that once we know what past history points need to be calculated, all other points in between will never need to be calculated. This then requires only enough memory to store the known and computed past intervals of the systems history, allowing the deallocation and recovery of much of the stored memory. This results in less memory requirements which translates into much faster computations.
\begin{definition} Define a parameter $\eta$ which determines the `reset interval'. This represents the number of points in the past history to store in the current time point set $U^k$ at each weight. \end{definition}
\begin{definition} Define weighting sets $W^i$ with elements $w^i$ ordered in the same manner as the sets $U^i$.  \end{definition}
\noindent \textbf{Algorithm 4.1.} Assume an adaptive memory time step power law scheme. Assume $w^0=1$. When there are more than $\eta$ points in the set $W^i$ of any given weight, define a subset of $W^i$ as the elements in $W^i$ of the given weight. Then from this subset, the weight of the least element is doubled, and the second lowest element is removed from $W^i$ altogether. The time marching scheme is given as follows:
\begin{eqnarray}
u^{k+1}_{j}&=&u^{k}_{j}+\alpha\frac{\Delta_t^{\gamma}}{\Delta_{x}^2}\sum_{\{u^i_j\in U^k\}}\psi(\gamma,k-i)w^i\delta^i_j.\nonumber \\
w^{k+1}&=&1 \nonumber \\
W^{k+1} &:=& \{w^i\in W^k\}+\{w^{k+1}\} \nonumber \\
U^{k+1}&:=&\{u^i\in U^k\}+\{u^{k+1}\} \nonumber \\
\end{eqnarray}
so that when the number $N$ elements of a given weight is $N>\eta$, the algorithm is condensed.

\subsection{Numerical implementation}
From an applied perspective, to keep a well ordered list of points in $U$ we make use of linked lists as the data structure. A linked list is one of the fundamental data structures used in computer programming. It consists of nodes that can contain data and their own data structures, and also pointers from each node to the next and/or previous nodes (Fig. \ref{linkedlists}A). This means one node can be removed from the set and deallocated from memory while the order of the set is maintained. 
In our case each node is an item $u$ of the set $U$ describing a time point necessary for the current time step, with the entire list representing all points $u$ that make up the current iteration $U^k$. Once a time point $u$ is no longer necessary for the simulation, it can be removed (Fig. \ref{linkedlists}B). New computed time points, $u^{k+1}$ are added into the list (Fig. \ref{linkedlists}C). The data structure is initialized as illustrated in Fig. \ref{linkedlists}D. At each time step, a new structure containing $u^{k+1}$ and $w^{k+1}$ is added onto the end of the list. When the elements of a specific weight $\omega$ grow larger than the limit $\eta$, the first two values of the weight $\omega$ are condensed into one value, with the weight doubling and one value being removed from the memory. We show as an example the fourth step in a simulation when $\eta=3$ (Fig. \ref{linkedlists}E).  As this process iterates in time, one has a condensed list of only the time points that will be needed for the rest of the simulation in the list. For example, Fig. \ref{linkedlists}F shows the transition to the $17^{th}$ step of a simulation.

\begin{figure}
\begin{center}
\includegraphics[width=6in]{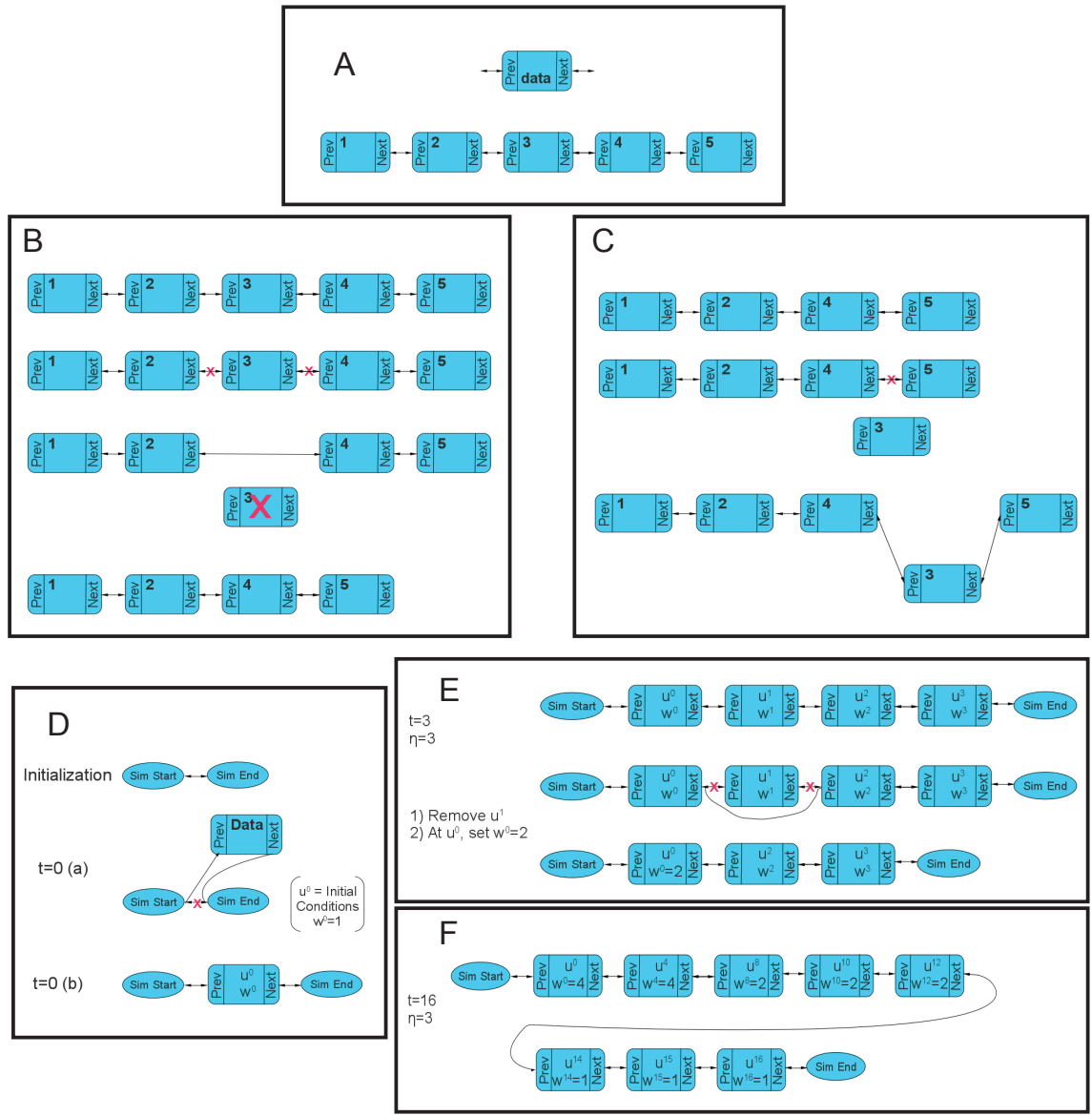}    
\caption{\footnotesize Overview of the use of linked list data structures for the algorithmic numerical implementation of adaptive memory time step as a power law scheme. See the text for details.} 
\label{linkedlists}
\end{center}
\end{figure}

\section{A smart adaptive time step memory algorithm }
\label{fracsmart}
Regular diffusion relies on the assumption of a Markovian process that is not necessarily present in natural systems. One approach to modeling these non-Markovian processes is using the fractional diffusion equation introduced above. Mathematically such methods have been around for decades but it is relatively recently that they have been applied to the natural sciences. Computationally efficient methods for numerically evaluating these equations are a necessity if fractional calculus models are to be applied to modeling real world physical and biological processes. It should be noted that while in this work the simulations were done in the subdiffusive regime for a simple point source, the methods we derive here are directly applicable to more complex sources or superdiffusion ($\gamma>1$). However, complex fast-evolving past histories in the form of a forcing function $(f(u))$ or oscillations generated in a superdiffusion regime will result in much larger errors for both the short and adaptive memory methods. In the case of the adaptive memory method introduced in this paper this is due to its `open-loop'-like algorithm that blindly increases the spacing between points in the summation as the calculation goes further back in time and which does not take into account the speed of evolution of the equation. Adaptive time approaches for regular differential equations often make the integration step size a function of the derivative, i.e., more closely spaced time steps are used when the function is oscillating quickly and cannot be averaged without loss of accuracy, and widely spaced time steps are used where the function is smooth. In the current implementation we have assumed that the past history function $\psi(\gamma,m)\delta^{k-m}_{i,j}$ is smooth. In this last section we extend the original implementation to develop a `smart' `closed-loop'-like algorithm where the step size of the backwards summation is dependent on the derivative of the past history function, i.e., a form of feedback. This optimizes the computational speed of the simulation while reducing the error due to the averaging of time points in the backwards summation, ultimately resulting in low errors for both high frequency forcing functions in the subdiffusion regime and for oscillations intrinsic to the superdiffusion regime. 

\subsection{Approximating the discrete Gr\"{u}nwald-Letnikov series as a continuous time integral }
Our analytical approach for smart adaptive memory makes use of the continuous Gr\"{u}nwald-Letnikov integral. In this form, we can then define a minimum threshold error function based on the derivative that ensures that no period in the history of the system is misrepresented up to the defined error. 

Recalling equation \ref{eq:frac_final}, the discretized form of the fractional diffusion equation, the summation term can be defined as a Riemann sum, which in the limit as $\Delta_t\rightarrow0$ approaches a continuous time integral. The benefits of an integral formulation is that the backwards integration would be separate from a defined time grid for the series, and higher order methods for integrating, such as Simpson's rule, could be used. This would be impossible in discrete time since it is necessary to project and interpolate between points on the grid.

Defining a Riemann sum $S$ over an interval $x_1$ to $x_n$,

\begin{equation}
 S=\sum^n_{i=0}g(y_i)(x_i-x_{i-1})
\end{equation}
there is a correspondence between the discrete summation and the integral
\begin{equation}\nonumber
\label{frac_discreteODE}
\frac{u^{k+1}_{j,l}-u^{k}}{\Delta_t}=\Delta_t^{\gamma-1}\sum_{m=0}^{t/\tau} \psi(\gamma,m)f(u_{k-m}).
\end{equation}
such that,
\begin{eqnarray}\nonumber
 g(y_i)&=&\psi(\gamma,m)f(u_{k-m})\nonumber \\
 i&=&m\nonumber \\
 n&=&t/\tau\nonumber \\ 
 x_i&=&m=0...n \nonumber
\end{eqnarray}
where the width of each segment is 1. As  $\Delta_t\rightarrow0$, this sum gets closer and closer to, and can be approximated by, the continuous time integral
\begin{equation}
 \int_{\tau=0}^t\psi(\gamma,\tau/\Delta_t)f(u(t-\tau)d\tau
\end{equation}
which allows the discretized version to be rewritten in continuous form as 
\begin{equation}
 \frac{u^{k+1}_{j,l}-u^{k}}{\Delta_t}=\Delta_t^{\gamma-1}\int_{\tau=0}^t\psi(\gamma,\tau/\Delta_t)f(u(t-\tau))d\tau
\label{grunlet_cont}
\end{equation}
The function $f(u(t-\tau))$ can be interpolated from the previously calculated values of $u$. The original definition of $\psi(\gamma,m)$ is only defined for $m\in \mathbb{N}_0$, and so needs to extend into the continuous domain. With an analytical continuous function representing $\psi$, one is then free to rediscretize the continuous integral in the most optimal way to solve the problem efficiently.

\subsection{Extension of $\psi(\gamma,m)$ to the positive real domain}

We are interested in a function $\Psi(\gamma,r)$ that is an extension of $\psi(\gamma,m)$ into the positive real domain and is defined such that for all $r\in \mathbb{N}_0$, $\psi(\gamma,r)$ = $\psi(\gamma,m)$. We begin with a basic linear interpolation of $\psi(\gamma,m)$ over non-integer space, and the result is shown in  Fig. \ref{psi_interp}A for various values of $\gamma$. While a linear interpolation provides a reasonable approximation of $\psi$, it is not a smooth function and it is a first order approximation that obviously does not work well for areas of $\psi$ that have a high second derivative (e.g., $r\ll1$, for $\gamma<1$). Since we don't have the ability to increase the number of points we are basing our interpolation on, we consider other options to obtain a more accurate and smoother approximation.
\begin{figure}
\begin{center}
\includegraphics[width=4in]{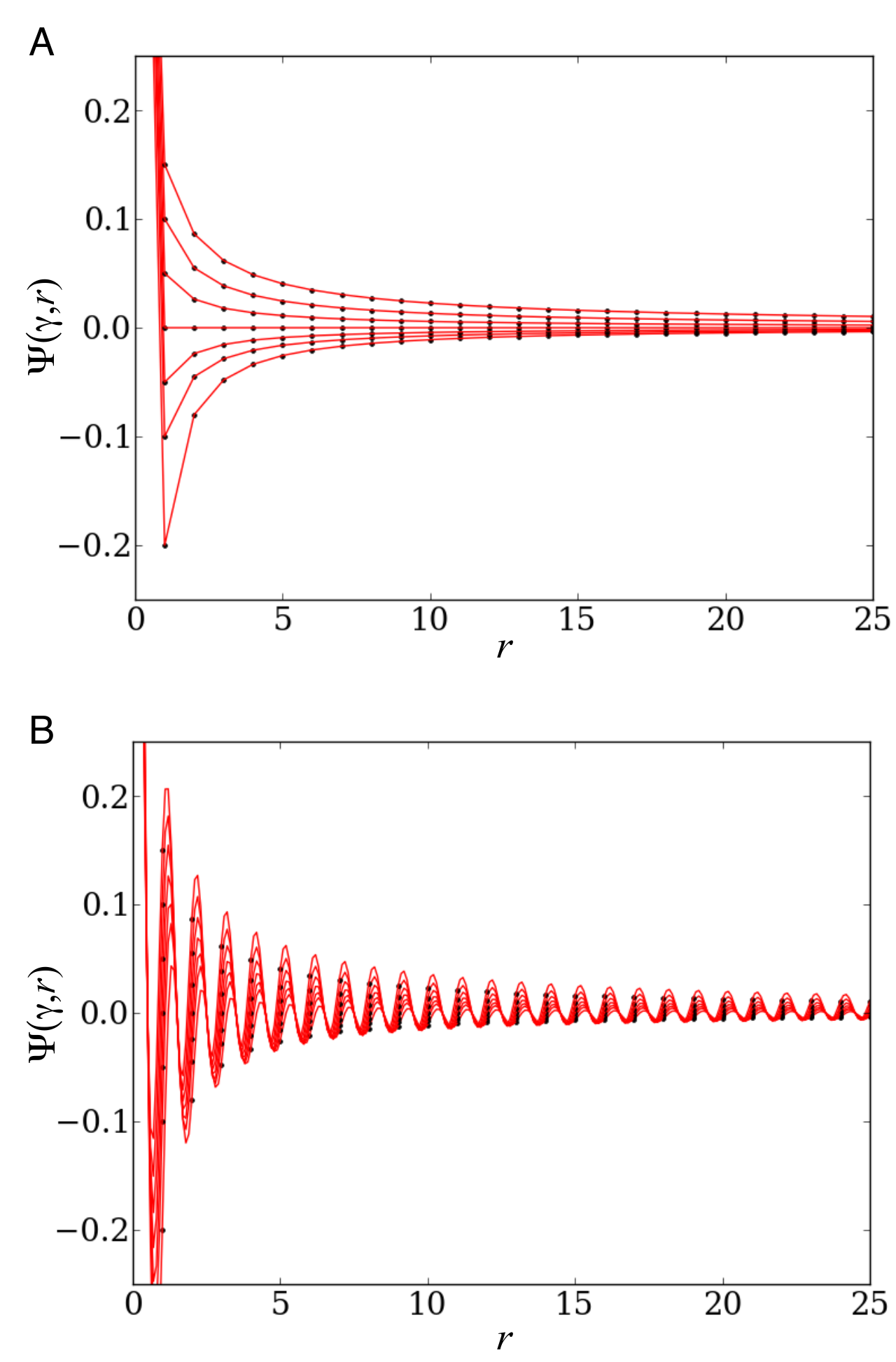}    
\caption{Computing a continuous version of $\psi(\gamma,m)$ for various values of $\gamma$. In all subplots, $\gamma$ = (.85, .90..., 1.10, 1.15) from the bottom trace to the top. Exact function $\psi$ is denoted by discrete points. A. $\Psi(\gamma,r)$ as a linear interpolation of the exact $\psi(\gamma,m)$. B. Extending the definitions of $\psi(\gamma,m)$ to all $r$ in the positive real domain.}
\label{psi_interp}
\end{center}
\end{figure}

Next, we consider simply extending the original definition of $\psi(\gamma,m)$ and expanding it to all points $r$ by using the gamma function to re-express the factorial. With $r$ extending to non-integer space in the positive real domain, and with the $(-1)^r$ term, we get an oscillating complex function. A plot of the real part of the function,
\begin{equation}
 \Psi(\gamma,r)=Real\left\{\frac{(-1)^r\Gamma(2-\gamma)}{\Gamma(r+1)\Gamma(2-\gamma-r)}\right\},
\label{psifunc2}
\end{equation}
shows that this method results in a poor approximation of $\psi$ due to the oscillatory behavior (Fig. \ref{psi_interp}B.)

Another possibility for deriving a continuous function $\Psi$ is to consider a rational polynomial function. One can rewrite the recursive series $\psi(\gamma,m)$ as a truncated approximation using the relation from equation \ref{eq:iterativepsi}:

\begin{eqnarray}
 \psi(\gamma,m)&=&-\psi(\gamma,m-1)\frac{2-\gamma-m}{m}\nonumber \\
 \psi(\gamma,0)&=&1\nonumber
\end{eqnarray}
This can be written in terms of a finite product as
\begin{align}
 & \psi(\gamma,m)=\frac{(-1)^m\Gamma(2-\gamma)}{m!\Gamma(2-\gamma-m)}= \cdots \\&-\psi(\gamma,m-1)\frac{2-\gamma-m}{m}=\prod^m_{n=1}\left[\frac{\gamma+n-2}{n}\right]=\prod^m_{n=1}\left[1+\frac{\gamma-2}{n}\right] \nonumber
\end{align}
Then, using a transform to show that an infinite rational function will intersect all these points, define a rational function so that 

\begin{equation}
\Psi(\gamma,r)=R_{\alpha,\beta}=\frac{P_\alpha}{Q_\beta}
\end{equation}
where 
\begin{equation}
 R_{\alpha,\beta}=\frac{p_0 + p_1r + p_2r^2...p_\alpha r^\alpha}{q_0 + q_1r + q_2r^2...q_\beta r^\beta}. 
\label{rational_poly}
\end{equation}
As $(\alpha,\beta)\rightarrow\infty$, $\Psi(\gamma,r)$ will approach $\psi(\gamma,m)$ for all $r\in \mathbb{N}_0$.

An infinite rational expression, however,would be too costly to compute. One can truncate the expression however, and get a closed-form expression with a very close fit that approaches the analytical recursive $\psi(\gamma,m)$. As $m\rightarrow\infty$, $\psi\rightarrow0$, which implies that $\beta>\alpha$ for this truncation. 

Given the number of coefficients $p_0,...,p_\alpha, q_0,...,q_\beta $ and flexibility in choosing $\alpha$ and $\beta$, there are multiple possible solutions to equation \ref{rational_poly}. But for all cases we consider the basic set of constraints given by the system of equations:

\begin{eqnarray}
 \psi(\gamma,0)&=&\Psi(\gamma,0)\nonumber \\
 \psi(\gamma,1)&=&\Psi(\gamma,1)\nonumber \\
 \psi(\gamma,2)&=&\Psi(\gamma,2)\nonumber \\
&\dots &\nonumber \\
 \psi(\gamma,M)&=&\Psi(\gamma,M)
\end{eqnarray}
where $M$ is an integer. 

The values of $\alpha$ and $\beta$ are also important considerations, since higher order polynomials will yield a more accurate approximation. Although the resulting function $\Psi$ will increase in complexity along with accuracy, a powerful advantage of a $\psi\rightarrow\Psi$ transform that uses a finite rational polynomial function is that we will obtain a continuous version of the recursive function $\psi$, that is a closed-form expression.

No matter the exact method or transform used to obtain the new function $\Psi(\gamma,r)$, once derived, we can then directly insert it into the numerical implementation. We can drop all points in the past history function ($\Psi(\gamma,m)f(u_{k-m}$)) in the regions where the second derivative is below a certain threshold (i.e., where the function is slowly changing), and integrate the function on the resultant mesh using equation \ref{grunlet_cont}, with the substitution of the continuous $\Psi(\gamma,r)$, with $r=\tau/\Delta_t$:

\begin{equation}
 \frac{u^{k+1}_{j,l}-u^{k}}{\Delta_t}=\Delta_t^{\gamma-1}\int_{\tau=0}^t\Psi(\gamma,\tau/\Delta_t)f(u(t-\tau))d\tau
\end{equation}

As discussed before, this smart adaptive step extension to the original algorithm will allow us to minimize errors by using smaller integration step sizes when the past history function is quickly changing due to a fast-evolving external forcing function to the system, or oscillating behavior integral to the dynamics of the system itself. In addition, this approach will be able to take advantage of the large body of existing literature and numerical methods solutions packages for solving continuous equations.

\section*{Acknowledgments}
The authors wish to thank Dr. Richard Magin  and Dr. Igor Podlubny for helpful discussions during the course of this work. This work was supported by NIH NS054736 and ARO 63795EGII.





\bibliographystyle{amsplain}
\bibliography{MacDonaldarXiv2015}

\providecommand{\bysame}{\leavevmode\hbox to3em{\hrulefill}\thinspace}
\providecommand{\MR}{\relax\ifhmode\unskip\space\fi MR }
\providecommand{\MRhref}[2]{%
  \href{http://www.ams.org/mathscinet-getitem?mr=#1}{#2}
}
\providecommand{\href}[2]{#2}
\begin{thebibliography}{10}

\bibitem{baeumer2001subordinated}
B.~Baeumer, D.A. Benson, M.M. Meerschaert, and S.W. Wheatcraft,
  \emph{{Subordinated advection-dispersion equation for contaminant
  transport}}, Water Resources Research \textbf{37} (2001), no.~6, 1543--1550.

\bibitem{chen2008finite}
C.~Chen, F.~Liu, and K.~Burrage, \emph{{Finite difference methods and a Fourier
  analysis for the fractional reaction--subdiffusion equation}}, Applied
  Mathematics and Computation \textbf{198} (2008), no.~2, 754--769.

\bibitem{chen2009numerical}
C.M. Chen, F.~Liu, I.~Turner, and V.~Anh, \emph{{Numerical schemes and
  multivariate extrapolation of a two-dimensional anomalous sub-diffusion
  equation}}, Numerical Algorithms (2009), 1--21.

\bibitem{PhysRevE.55.6821}
Albert Compte, \emph{Continuous time random walks on moving fluids}, Phys. Rev.
  E \textbf{55} (1997), no.~6, 6821--6831.

\bibitem{cushman36fractional}
J.H. Cushman and T.R. Ginn, \emph{{Fractional advection-dispersion equation: A
  classical mass balance with convolution-Fickian flux}}, Water resources
  research \textbf{36} (2000), no.~12, 3763--66.

\bibitem{fellah2002application}
ZEA Fellah, C.~Depollier, and M.~Fellah, \emph{{Application of fractional
  calculus to the sound waves propagation in rigid porous materials: Validation
  via ultrasonic measurements}}, Acta Acustica United With Acustica \textbf{88}
  (2002), 34--39.

\bibitem{ford2001numerical}
Neville~J. Ford and A.~Charles Simpson, \emph{The numerical solution of
  fractional differential equations: Speed versus accuracy}, Numerical
  Algorithms \textbf{26} (2001), 333--346.

\bibitem{gorenflo2002time}
R.~Gorenflo, F.~Mainardi, D.~Moretti, and P.~Paradisi, \emph{{Time fractional
  diffusion: a discrete random walk approach}}, Nonlinear Dynamics \textbf{29}
  (2002), no.~1, 129--143.

\bibitem{henry2008fractional}
BI~Henry, TAM Langlands, and SL~Wearne, \emph{{Fractional cable models for
  spiny neuronal dendrites}}, Physical review letters \textbf{100} (2008),
  no.~12, 128103.

\bibitem{langlands2005accuracy}
TAM Langlands and BI~Henry, \emph{{The accuracy and stability of an implicit
  solution method for the fractional diffusion equation}}, Journal of
  Computational Physics \textbf{205} (2005), no.~2, 719--736.

\bibitem{Li-FastTimeStepFracIntegrals2010}
Jing-Rebecca Li, \emph{A fast time stepping method for evaluating fractional
  integrals}, SIAM Journal on Scientific Computing \textbf{31} (2010), no.~6,
  4696--4714.

\bibitem{liu2007stability}
F.~Liu, P.~Zhuang, V.~Anh, I.~Turner, and K.~Burrage, \emph{{Stability and
  convergence of the difference methods for the space--time fractional
  advection--diffusion equation}}, Applied Mathematics and Computation
  \textbf{191} (2007), no.~1, 12--20.

\bibitem{Lubich-AdaptiveFastOblConvEquations2008}
Mar{\'i}a L{\'o}pez-Fern{\'a}ndez, Christian Lubich, and Achim Sch\"{a}dle,
  \emph{Adaptive, fast, and oblivious convolution in evolution equations with
  memory}, SIAM Journal on Scientific Computing \textbf{30} (2008), no.~2,
  1015--1037.

\bibitem{Lubich-ConvQuadDiscOperationalCalcI}
C.~Lubich, \emph{Convolution quadrature and discretized operational calculus
  i}, Numerische Mathematik \textbf{52} (1988), no.~2, 129--145.

\bibitem{Lubich-disc-fraccal1986}
Ch. Lubich, \emph{Discretized fractional calculus}, SIAM Journal on
  Mathematical Analysis \textbf{17} (1986), no.~3, 704--719.

\bibitem{Lubich-ConvQuadDiscOperationalCalcII}
Christian Lubich, \emph{Convolution quadrature and discretized operational
  calculus ii}, Numerische Mathematik \textbf{52} (1988), 413--425.

\bibitem{Lubich-NonreflectingBoundary2002}
Christian Lubich and Achim Sch\"{a}dle, \emph{Fast convolution for
  nonreflecting boundary conditions}, SIAM Journal on Scientific Computing
  \textbf{24} (2002), no.~1, 161--182.

\bibitem{Magin:2004p3137}
R~Magin, \emph{{Fractional calculus in bioengineering.}}, Critical reviews in
  biomedical engineering (2004).

\bibitem{Magin:2006uo}
Richard~L Magin, \emph{{Fractional Calculus in Bioengineering}}, Begell House
  Publishers, jan 2006.

\bibitem{magin2008modeling}
RL~Magin and M.~Ovadia, \emph{{Modeling the Cardiac Tissue Electrode Interface
  Using Fractional Calculus}}, Journal of Vibration and Control \textbf{14}
  (2008), no.~9-10, 1431.

\bibitem{mainardi2006sub}
F.~Mainardi, A.~Mura, G.~Pagnini, and R.~Gorenflo, \emph{{Sub-diffusion
  equations of fractional order and their fundamental solutions}}, Mathematical
  Methods in Engineering (2006), 20--48.

\bibitem{mathieu2003fractional}
B.~Mathieu, P.~Melchior, A.~Oustaloup, and C.~Ceyral, \emph{{Fractional
  differentiation for edge detection}}, Signal Processing \textbf{83} (2003),
  no.~11, 2421--2432.

\bibitem{mclean2009convergence}
W.~McLean and K.~Mustapha, \emph{{Convergence analysis of a discontinuous
  Galerkin method for a sub-diffusion equation}}, Numerical Algorithms
  \textbf{52} (2009), no.~1, 69--88.

\bibitem{Metea:2006bz}
Monica~R Metea and Eric~A Newman, \emph{{Calcium signaling in specialized glial
  cells}}, Glia \textbf{54} (2006), no.~7, 650--655.

\bibitem{metzler1999anomalous}
R.~Metzler, E.~Barkai, and J.~Klafter, \emph{{Anomalous diffusion and
  relaxation close to thermal equilibrium: A fractional Fokker-Planck equation
  approach}}, Physical Review Letters \textbf{82} (1999), no.~18, 3563--3567.

\bibitem{metzler2000random}
R.~Metzler and J.~Klafter, \emph{{The random walk's guide to anomalous
  diffusion: a fractional dynamics approach}}, Physics Reports \textbf{339}
  (2000), no.~1, 1--77.

\bibitem{Newman:2001p1009}
E~A Newman, \emph{{Propagation of intercellular calcium waves in retinal
  astrocytes and M{\"u}ller cells}}, The Journal of neuroscience : the official
  journal of the Society for Neuroscience \textbf{21} (2001), no.~7,
  2215--2223.

\bibitem{podlubny1999fractional}
I.~Podlubny, \emph{{Fractional differential equations}}, Academic Press New
  York, 1999.

\bibitem{podlubny2009matrix}
I.~Podlubny, A.~Chechkin, T.~Skovranek, Y.Q. Chen, and B.M. Vinagre~Jara,
  \emph{{Matrix approach to discrete fractional calculus II: partial fractional
  differential equations}}, Journal of Computational Physics \textbf{228}
  (2009), no.~8, 3137--3153.

\bibitem{press1992numerical}
W.H. Press, S.A. Teukolsky, W.T. Vetterling, and B.P. Flannery,
  \emph{{Numerical recipes in C}}, Cambridge Univ. Press Cambridge MA, USA:,
  1992.

\bibitem{Lubich-FastOblConvQuad2006}
Achim Sch\"{a}dle, Mar{\'i}a L{\'o}pez-Fern{\'a}ndez, and Christian Lubich,
  \emph{Fast and oblivious convolution quadrature}, SIAM Journal on Scientific
  Computing \textbf{28} (2006), no.~2, 421--438.

\bibitem{sebaa2006application}
N.~Sebaa, ZEA Fellah, W.~Lauriks, and C.~Depollier, \emph{{Application of
  fractional calculus to ultrasonic wave propagation in human cancellous
  bone}}, Signal Processing \textbf{86} (2006), no.~10, 2668--2677.

\bibitem{seki2003fractional}
K.~Seki, M.~Wojcik, and M.~Tachiya, \emph{{Fractional reaction-diffusion
  equation}}, The Journal of Chemical Physics \textbf{119} (2003), 2165.

\bibitem{selhuber2009variety}
C.~Selhuber-Unkel, P.~Yde, K.~Berg-S{\o}rensen, and L.B. Oddershede,
  \emph{{Variety in intracellular diffusion during the cell cycle}}, Physical
  Biology \textbf{6} (2009), 025015.

\bibitem{soczkffiwicz2002application}
E.~Soczkiewicz, B.~Krzywoustego, and P.~Gliwice, \emph{{Application of
  Fractional Calculus in the Theory of Viscoelasticity}}, Molecular and Quantum
  Acoustics \textbf{23} (2002), 397.

\bibitem{volterra1931}
V.~Volterra, \emph{{Le{\c{c}}ons sur la th{\'e}orie math{\'e}matique de la
  lutte pour la vie}}, Paris, France (1931).

\bibitem{Yu:2009p3681}
D~Yu, M~Buibas, SK~Chow, IY~Lee, Z~Singer, and GA~Silva,
  \emph{{Characterization of Calcium-Mediated Intracellular and Intercellular
  Signaling in the rMC-1 Glial Cell Line}}, Cellular and Molecular
  Bioengineering \textbf{2} (2009), no.~1, 144--155.

\bibitem{yuste2003explicit}
SB~Yuste and L.~Acedo, \emph{{On an explicit finite difference method for
  fractional diffusion equations}}, Arxiv preprint cs/0311011 (2003).

\end{thebibliography}

\end{document}